\documentclass{amsart}
\usepackage[utf8]{inputenc}
\usepackage{amsmath,amsfonts,amssymb,mathrsfs,amstext,amscd,latexsym,amsthm, mathtools}
\usepackage{hyperref}
\usepackage{xcolor}
\usepackage{pgfplots}
\usepgfplotslibrary{fillbetween}
\usepackage{graphicx}
\usepackage{subcaption}

\def\R{\mathbb{R}}
\def\e{\mathbf{e}}

\newtheorem{theorem}{Theorem}[section]
\newtheorem{lemma}[theorem]{Lemma}

\newtheorem{proposition}[theorem]{Proposition}
\newtheorem*{definition}{Definition}
\newtheorem*{remark}{Remark}

\DeclareMathOperator{\fff}{I}
\DeclareMathOperator{\sff}{II}
\DeclareMathOperator{\graph}{graph}

\title[Translating solutions to extrinsic geometric flows]{Rotationally symmetric translating solutions to extrinsic geometric flows}
\author{Sathyanarayanan Rengaswami}
\date{October 2020}

\usepackage{graphicx}

\begin{document}

\begin{abstract}
{Analogous to the bowl soliton of mean curvature flow, we construct rotationally symmetric translating solutions to a very large class of extrinsic curvature flows, namely those whose speeds are $\alpha$-homogeneous ($\alpha>0$), elliptic and symmetric with respect to the principal curvatures. We show that these solutions are necessarily convex, and give precise criteria for the speed functions which determine whether these translators are defined on all of $\mathbb{R}^{n}$ or contained in a cylinder. For speeds that are nonzero when at least one of the principal curvatures is nonzero, we are also able to describe the asymptotics of the translator at infinity.}
\end{abstract}

\maketitle

\setcounter{tocdepth}{1}

\tableofcontents

\section{Introduction}

Ancient solutions to geometric flows, i.e. solutions that are defined on a time interval of the form $(-\infty, T), -\infty<T \leq\infty$, are fundamental to understanding the global behaviour of these flows, because they arise as limits of rescalings near singularities \cite{MR1375255}. An important class of singularity models for extrinsic geometric flows are the translating solutions, so named because they evolve by ambient translation with constant velocity. Translating solutions arise in the analysis of singularities directly, as blow-up limits \cite{AngenentVelazquez,MR1316556}, and also indirectly, in the sense that convex ancient solutions tend to decompose into configurations of asymptotic translators \cite{ADS2,Betal,BLTatomic,BLT1,MR4127403,ChoiChoiDask}. In some cases, it can be shown that translating blow-up limits are necessarily rotationally symmetric \cite{MR3562950,MR3375531}.

In the present paper, we are interested in translating solutions to a very large class of extrinsic curvature flows. In particular, we shall prove the existence (and uniqueness) of \emph{bowl-type solitons}, that is, complete rotational graphs of class $C^2$ over either $\mathbb{R}^n$ or a ball $B_R$ of radius $R$ centred at the origin, with zero height and gradient at the origin. For a large class of speeds we also obtain an asymptotic expansion at (spatial) infinity.

Consider the extrinsic geometric flow
\begin{equation} \label{flow}
  \frac{\partial \Vec{X}}{\partial t}= -f\Vec{N}  
\end{equation}
of a one-parameter family $\Vec{X}:M^n\times I\to\R^{n+1}$ of immersed hypersurfaces, where $\Vec{N}$ denotes the (outward) unit normal field and the speed function $f$ is given by a function $f(\kappa_1,\dots\kappa_n)$ of the principal curvatures $\kappa_1\le\dots\le\kappa_n$ which is a positive function defined for positive principal curvatures, monotone increasing in each variable and $\alpha$-homogeneous, i.e. $f(\lambda z)= \lambda^\alpha f(z)$ for all $\lambda>0, z\in \mathbb{R}^n$. 

We are interested in the following questions: When do bowl-type solitons exist for this flow? When are these defined on all of $\mathbb{R}^n$, and when are they defined over bounded domains? What can we say about the asymptotics at infinity? In the case where $f$ is the mean curvature, Altschuler and Wu \cite{AW} proved that bowl-type solitons do exist. 
Their approach was to use elliptic PDE theory on general domains and deduce the rotationally symmetric case as a corollary. Later, Clutterbuck, Schn\"urer and Schulze in \cite{CSS} approached the rotational case directly, studying the corresponding ODE initial value problem, and were able to give precise asymptotics of the solution at infinity. Their approach was to use upper and lower barriers to the ODE solution. Note that this provides an alternative proof of the existence of the bowl soliton: by constructing a suitable lower barrier, one can solve the Dirichlet problem over a small ball, and then extend the solution uniquely using the ODE analysis. This is the approach taken in \cite[Theorem 13.38]{EGF}. Urbas \cite{Urb} studied soliton solutions (including bowl-type solitons) to flows by powers of the Gauss curvature, exploiting techniques from the study of Monge--Amp\`ere-type equations. {Santaella \cite{Santaella1,Santaella2} studied translating solutions to flows by ratios $Q_k=S_{k+1}/S_k$ of consecutive elementary symmetric polynomials $S_k$ in the principal curvatures. In particular, he constructed a bowl-type soliton for the $Q_{n-1}$-flow (which is perhaps better known as the \emph{harmonic mean curvature flow}).} But a study of the general situation of homogeneous flow speeds has not yet been carried out to our knowledge (even under additional concavity assumptions on the speed). Indeed, all prior work focuses on the analysis of specific speed functions. Moreover, in the case of Gauss curvature flows, the analysis is made simple due to the fact that the resulting ODE is separable. We will show here that, in fact, general properties of the speed function already carry a lot of information about the shape of the bowl-type soliton.

We use an approach similar to that of \cite{CSS}: we seek a function $u: I \to \mathbb{R}$ (where $I=[0,R), 0<R\leq \infty$ ) of class $C^2$ such that the graph $y=u(|x|)$ in $\mathbb{R}^{n+1}$ is a translating solution to \eqref{flow}, derive the ODE it should satisfy, and analyse it to prove that such solutions do exist under certain (very general) hypotheses on $f$. We also use ODE analysis to work out the asymptotics of $u$:. In contrast to \cite{CSS} (or \cite{EGF}), we also prove directly, by analysis of the translator ODE, that the solution is of class $C^2$ up to the origin and smooth elsewhere. {PDE theory is only invoked, in the form of a basic lemma, to infer smoothness at the origin.}  

For the purposes of the present paper, we make the following definition. 

\begin{definition}
A $C^1$ function $f:\Omega \subset \mathbb{R}^{n} \to [0,\infty)$ such that $\Omega$ contains $\Gamma^n
_+ \doteqdot\{x \in \mathbb{R}^{n} :z_i>0\}$, is an \emph{admissible speed function} if
\begin{itemize}
    \item $f$ is invariant under permutation of its variables $z_1,\dots,z_n$. (Symmetry)
    \item $\frac{\partial f}{\partial z_i}>0$ for each $i=1,\dots,n$. (Ellipticity)
    \item $f$ is $\alpha$-homogeneous for some $\alpha>0$. (Homogeneity)
\end{itemize}
\end{definition}

This is a very large class of speeds. It allows, for example, the commonly studied class of flows by positive powers of one-homogeneous roots of ratios of elementary symmetric polynomials in the principal curvatures \cite{An07}. Note, moreover, that no concavity or smoothness conditions beyond $C^1$ are required.

Define $\mathbf{e} \doteqdot (1,...,1) \in \mathbb{R}^{n-1}$. We have the following dichotomy for the behaviour of $f$ at the boundary of the positive cone $\Gamma_+$: either $f(0,\mathbf{e})>0$ or $f(0,\mathbf{e})=0$. We call such speeds $nondegenerate$ and $degenerate$ respectively. Note that, although our domains do not necessarily include the point $(0,\mathbf{e})$, we can define $f(0,\mathbf{e}) \doteqdot \lim_{s \to 0}f(s,\mathbf{e})$, since $f$ is increasing in all arguments and nonnegative. Observe that this degeneracy condition corresponds to whether or not the cylinder $\mathbf{S}^{n-1} \times \mathbb{R}$ is a stationary solution of the corresponding flow.

In the nondegenerate case, cylinders collapse into a line in finite time. Therefore, one does not expect complete translators to be contained in a cylinder, and one would expect bowl-type solitons to be entire. In the mean curvature case, it was proved in \cite{CSS} that the asymptotic expansion of the bowl soliton up to order $|x|^2$ is $u(|x|)=\frac{
|x|^2}{2(n-1)}+o(|x|^2)$ as $|x| \to \infty$. Our main result shows that this is indeed the case for any admissible nondegenerate speed and that such an asymptotic expansion does hold.

\begin{theorem}\label{thm:entire}
Let $f$ be an admissible speed function. There exists a unique bowl-type soliton for the corresponding flow. It is the graph of a convex radial function $y=u(|x|)$, where $u\in C^2([0,R))$, $R\in(0,\infty]$.
If $f$ is of class $C^{k,\alpha}$, $\alpha\in(0,1)$, then $u\in C^{k+2,\alpha}([0,R))$. If $f$ is non-degenerate, then $R=\infty$ and
\[u(|x|)=C|x|^{\alpha+1}+o(|x|^{\alpha+1})\;\;\text{as}\;\; \vert x\vert\to\infty\,,\]
where $C\doteqdot \frac{1}{(\alpha+1)f(0,\mathbf{e})}$.
\end{theorem}

{It is somewhat surprising that the asymptotic expansion in Theorem \ref{thm:entire} is available under such general hypotheses on the speed. Indeed, the computation of lower order terms seems to be much more dependent on the specific form of $f$.}

Our next theorem concerns low homogeneities (regardless of whether they are degenerate or not.) 
\begin{theorem}\label{thm:low homogeneity}
If $f$ is an admissible speed function with $\alpha \leq 1/2$, then the corresponding bowl-type soliton is entire.
\end{theorem}

For higher homogeneities, we have examples of both possibilities (entireness or non-entireness of the corresponding bowl-type soliton), even within the ``mean curvature type'' setting $\alpha=1$. For example, when $f(\kappa_1,\dots,\kappa_n)=(\Sigma_{i=1}^n\kappa_i^{-1})^{-1}$ (harmonic mean curvature), the solution is defined on a ball, whereas when $f(\kappa_1,\kappa_2)=\sqrt{\kappa_1\kappa_2}$, which is the square root of the Gauss curvature, the solution is defined over $\mathbb{R}^2$. In the case where $f(\kappa_1,\kappa_2)=\kappa_1\kappa_2$ is the Gauss curvature, the solution is again defined over a ball. This difference is due to the behaviour of $x$ in the equation $f(x,y\mathbf{e})=1$ as $y\to \infty$. In the interest of completeness, we also provide a partial classification of this phenomenon. Note that since $f$ is strictly increasing with respect to each of the principal curvatures, the equation $f(x,y\mathbf{e})=1$ can be solved for $x$ in terms of $y$.

\begin{figure}
    \centering
    \begin{subfigure}[b]{0.32\textwidth}
        \begin{tikzpicture}
\begin{axis}[
    width=100,
    height=100,
    axis lines = left,
    xlabel = \(x\),
    ylabel = {\(y\)},
    xticklabels=\empty,
    yticklabels=\empty
]
\addplot [
    domain=0:1, 
    samples=100, 
    color=red,
]
{1-x};
\end{axis}
\end{tikzpicture}
        \caption{Mean curvature\\ \;}
        \label{fig:gull}
    \end{subfigure}
    \begin{subfigure}[b]{0.32\textwidth}
        \begin{tikzpicture}
\begin{axis}[
    width=100,
    height=100,
    axis lines = left,
    xlabel = \(x\),
    ylabel = {\(y\)},
    xticklabels=\empty,
    yticklabels=\empty
]
\addplot [
    domain=0:1, 
    samples=100, 
    color=red,
]
{1/x};
\end{axis}
\end{tikzpicture}
        \caption{Gauss curvature\\ \;}
        \label{fig:tiger}
    \end{subfigure}   
    \begin{subfigure}[b]{0.32\textwidth}
        \begin{tikzpicture}
\begin{axis}[
    width=100,
    height=100,
    axis lines = left,
    xlabel = \(x\),
    ylabel = {\(y\)},
    xticklabels=\empty,
    yticklabels=\empty
]
\addplot [
    domain=0:10, 
    samples=100, 
    color=black,
]
{0};
\addplot [
    dashed,
    domain=0:10, 
    samples=100, 
    color=black,
]
{1};
\addplot [
    dashed, 
    samples=100, 
    smooth,
    domain=0:6,
    black] coordinates {(1,0)(1,15)};
\addplot [
    domain=1:10, 
    samples=100, 
    color=red,
]
{x/(x-1)};
\end{axis}
\end{tikzpicture}
        \caption{Harmonic mean curvature}
        \label{fig:mouse}
    \end{subfigure}
    \caption{Level sets of the expression $f(x,y\mathbf{e})=1$.}\label{fig:animals}
\end{figure}
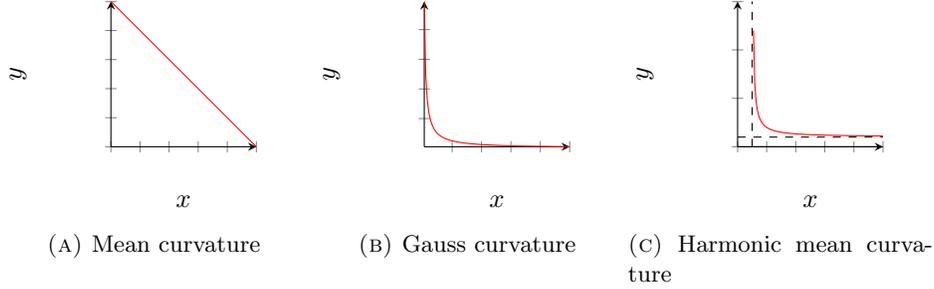

\begin{theorem}\label{L>0}
Let $f$ be an admissible speed function with $\alpha>1/2$ and $f(0,\mathbf{e})=0$. Consider the constraint equation $f(x,y\mathbf{e})=1$. If $x \to L>0$ as $y \to \infty$, then the bowl-type soliton corresponding to $f$ is defined on a ball $B$ and is asymptotic to the cylinder $\partial B \times \mathbb{R}$.
\end{theorem}

We may ask what happens when $L=0$. In this case, whether the solution is entire or over a bounded domain is determined by the rate at which $x$ decays to $0$ as $y\to \infty$. Recall that $f(t)=O(g(t))$ as $t \to \infty$ if there exists $C>0$ such that, for sufficiently large $t$,  $|f(t)/g(t)| \leq C$.

\begin{theorem}\label{L>0 refined}
Let $f$ be an admissible speed function with $\alpha>1/2$ and $f(0,\mathbf{e})=0$. Consider the constraint equation $f(x,y\mathbf{e})=1$. Suppose that $x \to 0$ as $y \to \infty$.
\begin{itemize}
\item[(i)] If $x=O(y^{-(2\alpha-1)})$ then the corresponding bowl-type soliton is entire.
\item[(ii)] If there exist constants $C>0,k \in (0,2\alpha-1)$ such that $x \geq C y^{-k}$ for sufficiently large $y$, then the corresponding bowl-type soliton is defined over a bounded domain.
\end{itemize}
\end{theorem}
We note here that this theorem does not describe what happens in case $x=O(y^{-k})$ for all $k<2\alpha-1$ but $x\neq O(y^{-(2\alpha-1)})$, for example $x=y^{-2(\alpha-1)}\log y$. However one usually does not encounter such extreme cases, as $f$ is typically an algebraic combination of the principal curvatures in most applications.

We will be deducing the above theorems as consequences of the analysis of the \emph{translator ODE}, which is an equation of the form

\begin{equation*}
    v'=(1+v^2)^{1+\beta}g\left(\frac{v}{r(1+v^2)^\beta},1\right)\,,
\end{equation*}
where $v$ refers to the slope of the profile curve, $\beta$ depends on $\alpha$ and the function $g$ depends on $f$. Since we seek bowl-type solitons, we impose the initial condition $v(0)=0$. As $(r,v)=(0,0)$ is not in the domain of the right hand side, this equation has a coordinate singularity at the origin, and hence its solvability near the origin does not follow from standard ODE theory. However, away from the origin, local solvability of the equation does follow from the Picard--Lindel{\"o}f theorem, and extensibility of the solution to its maximal interval follows from the existence of subsolutions and supersolutions to the ODE. (See \S \ref{sec:prelims} for a review of the requisite ODE theory.) We use this fact to construct a sequence of solutions to the ODE whose initial values converge to $(0,0)$, and show that the limit of these solutions is indeed the unique $C^1-$solution to the ODE with initial condition $(0,0)$. When $f$ satisfies the hypotheses of Theorem \ref{thm:entire}, we show that $\lim_{r \to \infty}\frac{v(r)}{r(1+v(r)^2)^\beta}$ converges to something finite, and thereby deduce the asymptotic expansion of the bowl-type soliton as $|x| \to \infty$. In the other cases, we are able to infer whether the solution is entire or not by applying
Lemma \ref{ODElemma}.

\subsection*{Acknowledgements}
The author would like to thank Dr. Mathew Langford and Dr. Theodora Bourni of the University of Tennessee, Knoxville for suggesting this problem, providing guidance in both literature review and in technical aspects throughout the research phase, proofreading the arguments and offering countless hours of Zoom advising during a time that made in-person communication challenging. The author would also like to express his gratitude to the Office of Research and Engagement of the University of Tennessee for the funding they provided during Summer 2021. 

\section{Preliminaries}\label{sec:prelims}

\subsection{ODE initial value problems}
We review the basics of ODE theory and some of the tools used in this paper. We use \cite{ODE} as a reference for this theory. The form of ODE most amenable to analysis is
\begin{equation} \label{generalode}
  x'=f(t,x)\,.
\end{equation}
Here, $x$ is an unknown ($\R^n$-valued) function of $t$, and $f$ is some known function on some open subset $U \subset \R \times \R^n$, and $x'\doteqdot dx/dt$. For our purposes, $x$ will be real-valued. Solutions to ODE can be visualized and analysed qualitatively by using a \emph{direction field}, which is the vector field $V(t,x)=(1,f(t,x))$ that gives the slope of a solution at each point. Integral curves of the direction field correspond to solutions of the ODE.

A differentiable function $x_+(t)$ satisfying
\[x_+'(t)>f(t,x_+(t))\] is called a \emph{supersolution} to (\ref{generalode}). Similarly, a differentiable function $x_-(t)$ satisfying
\[x_-'(t)<f(t,x_-(t))\] is called a \emph{subsolution} to \eqref{generalode}.

\newtheorem{Lemma}{Lemma}
\begin{lemma}
Let $x_+(t)$, $x_-(t)$ be super, sub solutions of the differential equation $x'=f(t,x)$ on $[t_0,T)$ respectively. For every solution $x(t)$ on $[t_0,T)$ we have 
\[x(t) < x_+(t), t \in [t_0,T) \text{ whenever }  x(t_0) \leq x_+(t_0)\]
respectively
\[x_-(t) < x(t), t \in [t_0,T) \text{ whenever } x_-(t_0) \leq x(t_0)\]
\end{lemma}

\begin{remark}
    If one replaces strong inequality by weak inequality in the definitions of sub and supersolutions, one gets weak inequalities instead of strong ones in the above lemma.
\end{remark}

A problem of the form
\begin{equation}\label{prob}
x'=f(t,x), \hspace{0.5cm} x(t_0)=x_0
\end{equation}
is called an \emph{initial value problem}, or \emph{IVP}. 

We recall the fundamental local existence-uniqueness and extensibility results.

\begin{theorem} \label{existence}
Suppose $f \in C(U, \mathbb{R}^n$) where $U$ is an open subset of $\R \times \R^n$ and $(t_0,x_0)\in U$.  If $f$ is locally Lipschitz continuous in the second argument, uniformly with respect to the first, then there exists a unique local solution $\Bar{x}(t)\in C(I)$ of the IVP (\ref{prob}), where $I$ is some interval around $t_0$.

Moreover, if $f \in C^k(U,\mathbb{R}^n)$, then $\Bar{x}(t)\in C^{k+1}(I)$
\end{theorem}

\begin{theorem} \label{extensibility}
Let $\phi(t)$ be a solution of the IVP (\ref{prob}) defined on the interval $(t_-,t_+)$. Suppose there is a compact subset $[t_0,t_+] \times C \subset U$ such that $\phi(t_m) \in C$ for some sequence $t_m \in [t_0,t_+)$ converging to $t_+$ Then there exists an extension to the interval $(t_-,t_++\epsilon)$ for some $\epsilon>0$.

In particular, if there is such a compact set $C$ for every $t_+>0$ ($C$ might depend on $t_+$), then the solution exists for all $t>t_0$.

The analogous statement holds for an extension to $(t_--\epsilon,t_+)$
\end{theorem}

\begin{remark}
    The form $[t_0,t_+] \times C$ for the compact set can be relaxed. We can simply require some compact set $K \subset U$, such that the projection of $K$ onto the $t$-coordinate contains $[t_0,t_+]$, and that $(t_m, \phi(t_m))\in K$. 
\end{remark}

The above theorem is an extensibility result. For instance, it guarantees that if you can solve your ODE locally at any $t=t_0$, and have super and subsolutions that exist for all $t>t_0$, then your solution extends to all ``future times'' $t>t_0$.

The following theorem provides estimates for the difference between two solutions of an ODE. This is a special case of \cite[Theorem 2.8]{ODE}.

\begin{theorem} \label{ODEestimate}
    Suppose $f \in C(U, \R^n)$, $f=f(t,x)$ is Lipschitz continuous (with Lipschitz constant L) in the second argument, uniformly with respect to the first. If $x(t), y(t)$ are solutions of the respective IVPs
    \begin{align*}
        x'(t)=f(t,x), \, x(t_0)=x_0\\
        y'(t)=f(t,y), \, y(t_0)=y_0\,,
    \end{align*}
    then,
    \begin{equation*}
        |x(t)-y(t)| \leq |x_0-y_0| e^{L|t-t_0|}
    \end{equation*}
    for as long as both $x(t), y(t)$ are defined.
\end{theorem}

We also state a basic lemma which will be used to determine whether a given ODE blows up or not.

\begin{lemma} \label{ODElemma}
Consider the  problem
\[
\begin{split}
x'={}&x^\theta\\
x(t_0)={}&x_0>0\,.
\end{split}
\]
Its solution is
\begin{equation*}
x(t)=\left\{\begin{aligned}x_0+\log\left(\frac{t}{t_0}\right){}& \;\;\text{when}\;\; \theta=1\\
(x_0^{1-\theta}+(1-\theta)(t-t_0))^{\frac{1}{1-\theta}}{}&\;\;\text{when}\;\; \theta \neq 1\,.
\end{aligned}\right.
\end{equation*}
The solution exists for all $t>t_0$ if $\theta \leq 1$. It tends to infinity at $t=t_0+x_0/(\theta-1)$ if $\theta >1$.
\end{lemma}

\subsection{A PDE regularity lemma}
The following result provides higher regularity of $C^2$-solutions to elliptic PDE. It is a consequence of Schauder's estimate (see {\cite[Lemma 17.16]{GilbargTrudinger}}).

\begin{proposition}\label{smoothnessthm}
Suppose that $u \in C^2(\Omega)$ satisfies
\[
F(\cdot,u,Du,D^2u)=0 \;\; \text{in}\;\; \Omega\,,
\]
where $F:\Gamma\subset\Omega\times\R\times \R^n\times S^{n\times n}\to\R$ is monotone increasing with respect to the matrix variable. If $F\in C^{k,\alpha}(\Gamma)$ for some $k\ge 1$ and $0<\alpha<1$, then $u \in C^{k+2,\alpha}(\Omega)$. In particular, if $F$ is smooth, then so is $u$.
\end{proposition}

\subsection{The rotational translator ODE} \label{sec:TODE}

For a real-valued $C^2-$function $u$ of a single real variable, consider the graph of $y=u(|x|)$, where $x \in \R^n$. Its unit normal at a point $(x,u(|x|))$ is given by
\[
    \Vec{N}=\left(\frac{u'}{\sqrt{1+u'^2}}\frac{x}{|x|},\frac{-1}{\sqrt{1+u'^2}}\right) \in \R^n \times \R
\]
The principal curvatures of this hypersurface are $\kappa_1$, the curvature of the profile curve $y=u(r)$, and $\kappa_i$ $(i=2,...,n)$, the rotational curvatures, which are all the same. They are given by the equations
\begin{align*}
    \kappa_1&=\frac{u''}{(1+u'^2)^{3/2}}\\ \kappa_i&=\frac{u'}{r\sqrt{1+u'^2}}
\end{align*}
for $i=2,...,n$.

For the flow $\frac{\partial \Vec{X}}{\partial t }=-f \Vec{N}$, the equation for the standard translator with unit speed in the coordinate direction $e_{n+1}\doteqdot (0,...,0,1)\in \R^{n+1}$ is given by $\langle \Vec{N},e_{n+1}\rangle=-f$, (where $f$ is a function of principal curvatures at a point), from which one obtains
\begin{equation} \label{TODE}
    f\left(\frac{u''}{(1+u'^2)^{3/2}},\frac{u'}{r\sqrt{1+u'^2}},...,\frac{u'}{r\sqrt{1+u'^2}}\right)=\frac{1}{\sqrt{1+u'^2}}\,.
\end{equation}
We reduce the order and simplify notation by setting $v\doteqdot u'$ and $\mathbf{e}\doteqdot(1,...,1) \in \mathbb{R}^n$, so that
\begin{equation*}
    f\left(\frac{v'}{(1+v^2)^{3/2}},\frac{v}{r\sqrt{1+v^2}}\mathbf{e}\right)=\frac{1}{\sqrt{1+v^2}}\,.
\end{equation*}
The function $v$ has the geometric significance of being the gradient of the profile curve.

Now let $f$ be an admissible speed. Using the $\alpha$-homogeneity of $f$, one gets
\begin{equation*}
    f\left(\frac{v'}{(1+v^2)^{3/2-1/(2\alpha)}},\frac{v}{r(1+v^2)^{1/2-1/(2\alpha)}} \mathbf{e}\right)=1\,.
\end{equation*}

Define $h(x,y)\doteqdot f(x,y\mathbf{e})$. Since $f$ is strictly monotone in each argument, we can solve for $x$ using a unique function $g$, in the sense that $h(g(y,z),y)=z$. In the same way, we can solve for $y$ using a function $g_1$, i.e. $h(x,g_1(x,z))=z$. We remark here that since $f$ is an admissible speed, the implicit function theorem applies and hence $g, g_1$ are of class $C^1$. 

Thus, solving for $v'$ we get
\begin{equation*}
    v'=(1+v^2)^{\frac{3}{2}-\frac{1}{2\alpha}}g\left(\frac{v}{r(1+v^2)^{\frac{1}{2}-\frac{1}{2\alpha}}},1\right)\,.
\end{equation*}
We now set $\beta =\frac{1}{2}-\frac{1}{2\alpha}$, so that
\begin{equation}\label{eqn:alphaode}
    v'=(1+v^2)^{1+\beta}g\left(\frac{v}{r(1+v^2)^\beta},1\right)\,.
\end{equation}
We shall refer to this equation as the \emph{Translator ODE}. Note that since we seek bowl-type solitons, we have the initial condition $v(0)=u'(0)=0$. The value of $u(0)$ is arbitrary because the ODE is independent of $u$, but we can set $u(0)=0$ for convenience, so that our translator has its ``tip'' at the origin. One recovers $u$ from $v$ via the formula
\begin{equation} \label{uformula}
    u(r)= \int_0^r v(\rho) d \rho\,.
\end{equation}

\subsubsection{Level sets of the expression $\frac{w}{r(1+w^2)^\beta}$} \label{levset}

Let $m>0$ be a constant, and let $\frac{w}{m(1+w^2)^\beta}=r$.
Then
\begin{equation}\label{dr/dw}
    \frac{dr}{dw}=\frac{1}{m}\frac{1+(1-2\beta)w^2}{(1+w^2)^{1+\beta}}
\end{equation}
which means $r$ behaves very differently for $\beta>1/2$ and $\beta<1/2$. 
When $\beta>1/2$, $r$ is increasing for small values of $w$ and decreasing for large values, which means $r$ is not invertible. But if $\beta<1/2$, then  $dr/dw>0$, meaning $r$ is is increasing with respect to $w$ and hence, so is the inverse function which defines $w$ as a function of $r$. Figure \ref{fig:beta} illustrates this contrast. We consider only the latter case, since this corresponds to $\alpha>0$.

\begin{figure} 
\centering
\begin{tikzpicture}
\begin{axis}[
    axis lines = left,
    xlabel = \(v\),
    ylabel = {\(r\)},
    xticklabels=\empty,
    yticklabels=\empty
]
\addplot [
    domain=0:5, 
    samples=100, 
    color=red,
]
{x/(1+x^2)};
\addlegendentry{\(\beta>1/2\)}

\addplot [
    domain=0:5, 
    samples=100, 
    color=blue,
    ]
    {x/(1+x^2)^(1/4)};
\addlegendentry{\(0<\beta<1/2\)}
\addplot [
    domain=0:5, 
    samples=100, 
    color=green,
    ]
    {x/(1+x^2)^(-1/4)};
\addlegendentry{\(\beta<0\)}
\end{axis}
\end{tikzpicture}
\caption{The graph of $\frac{w}{m(1+w^2)^\beta}=r$}
\label{fig:beta}
\end{figure}
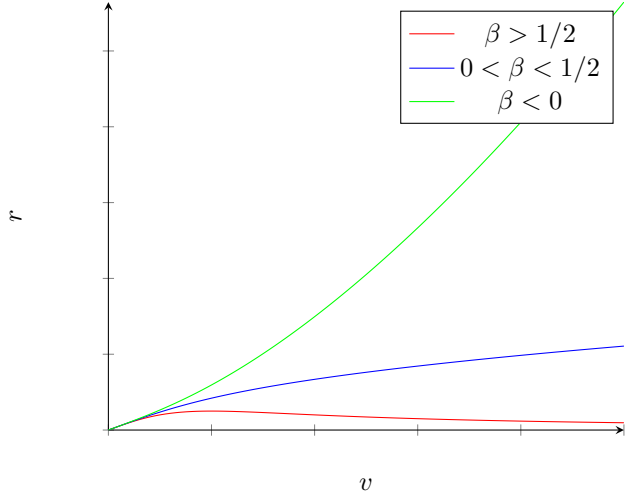

To make the notion of asymptotics precise, suppose $f,g$ are real valued functions of a real variable $t$. We say $f(t) \asymp g(t)$ as $t \to \infty$ if $f(t)/g(t) \to C$ for some constant $C \neq 0$. Then in our level set equation, $r(w) \asymp \frac{w^{1-2\beta}}{m}=\frac{w^{1/\alpha}}{m}$. Therefore, 
\begin{equation}
    w(r) \asymp (mr)^\alpha
\end{equation} 
And by equation (\ref{dr/dw}),
\[
\frac{dw}{dr}(r) \asymp m^\alpha w^{2\beta} \asymp m^\alpha r^{2\alpha \beta}=m^\alpha r^{\alpha-1}
\]
so that
\begin{equation} \label{dw/dr}
    \frac{dw}{dr} \asymp r^{\alpha-1}\,.
\end{equation}

For subsequent reference, we collect the following facts about the relation
\begin{equation*}
    \frac{w}{m(1+w^2)^\beta}=r.
\end{equation*}
The hypothesis $\beta<1/2$ is used for these lemmas.

\begin{lemma}\label{lem:level set invertibility}
    For constant $m>0$, $r$ is increasing with respect to $w$.
\end{lemma}

\begin{lemma}\label{lem:level set invertibility epsilon}
    For a fixed $r>0$, if we regard $m$ as a function of $w$, it is increasing with respect to $w$.
\end{lemma}
\begin{proof}
This is due to the symmetry of $r$ and $m$.
\end{proof}

\subsubsection{The domain of $g(\cdot,1)$}

Note that the values of a homogeneous function $f$ of $n$ variables are completely determined by the values of $f$ on $S^{n-1}$. Thus, our admissible speeds are completely determined in the positive cone $\Gamma^n_+$ by their values on $S^{n-1} \cap \Gamma_+^n$, the points on the unit sphere that have all positive coordinates. Further, since in our setting all inputs of $f$ but the first are equal, it suffices to consider points in $S^{n-1} \cap \Gamma^n_+$ of the form $(x,y\e)$, that is, $(x,y)\in \Gamma_+^2$ and $x^2+(n-1)y^2=1$. So we may describe $y$ as a function of $x$: $y(x)=\sqrt{\frac{1-x^2}{n-1}}$. Now we can describe the level set $f(x,y\e)=1$ as
\[
E=\left\{\frac{1}{f(x,y(x)\e)^{1/\alpha}}\left(x,y(x)\right): 0<x<1\right\}\, \subset \Gamma^2_+.
\]
From this point of view, we see that $E$ is a connected set, as it is the continuous image of an interval. From another point of view, $E$ can be identified with the graph of the function $g(\cdot,1)$: it is the set of points $(g(y,1),y)$ such that $y$ is in the domain $D\subset \R$ of $g(\cdot,1)$. $D$ is the projection of $E$ onto the $y$-coordinate, and hence it is connected. Now, due to the implicit function theorem, the domain of $g(\cdot,1)$ is an open interval, as we already know it is connected.
Since $\lim_{x \to 0} y(x)=1/\sqrt{n-1}$, we see that, in the degenerate case, $D$ is unbounded above. In the nondegenerate case, we see that $D$ is bounded above; here we find that $\sup D= 1/f(0,\e)^{1/\alpha}$, and we can extend $g$ continuously by defining $g(\sup D,1)=0$.

\subsection{Three examples}
We will illustrate the theorems with three examples, namely, flows by the harmonic mean curvature, scalar curvature, and powers of the Gauss curvature. These are defined, respectively, as $(\Sigma_{i=1}^n\kappa_i^{-1})^{-1}$, $\sqrt{\Sigma_{i<j} 2 \kappa_i \kappa_j}$, and $K^{\frac{\alpha}{n}}$ where $K=\Pi_1^n \kappa_i$ is the Gauss curvature. The former two are both 1-homogeneous functions, while the Gauss curvature flows can take any homogeneity $\alpha>0$. We note here that we are using a fact that will be proved in the next section, which is that (\ref{eqn:alphaode}) has a unique solution on some maximal interval $[0,R), 0<R\leq \infty$. These examples illustrate the motivation behind the formulation of all theorems of this paper.

\subsubsection{Harmonic mean curvature flow} 

Let us first study the harmonic mean curvature flow for $n \geq 2$. One finds that (\ref{eqn:alphaode}) becomes
\begin{equation}\label{hmcfv}
    v'=\frac{v}{v-(n-1)r}(1+v^2)\,.
\end{equation}

Observe that the function $v_-=nr$ is an ODE subsolution due to the following calculation:
\[
    v_-'=n \leq \frac{nr}{nr-(n-1)r}(1+n^2r^2)\,.
\]
Now, since the solution $v$ satisfies $v(0)=v_-(0)=0$, we have $v \geq v_-=nr$ for as long as $v$ exists. We then deduce that
\begin{align*}
     1 \leq \frac{v}{v-(n-1)r}&\leq n\,.
\end{align*}
Therefore
\begin{align*}
    &1      \leq \frac{v'}{1+v^2} \leq n \\
    \implies &r      \leq\tan^{-1}v        \leq nr \\
    \implies &\tan(r) \leq v                \leq \tan(nr)\,.
\end{align*}

Using (\ref{uformula}) to recover $u$, we see that the bowl-type soliton $y=u(|x|)$ is defined on a ball $B_R$ where $R\in [\frac{\pi}{2n},\frac{\pi}{2}]$. We also see it is asymptotic to the cylinder $\partial B_R \times \mathbb{R}$ since $v,v' \uparrow \infty$ as $r \uparrow R$, and therefore so do $u,u'$ as $|x| \uparrow R$.

\subsubsection{Scalar curvature flow}
When $n \geq 3$, the ODE corresponding to the scalar curvature flow is
\begin{equation}
    v'=(1+v^2)\frac{1-(n-1)(n-2)(v/r)^2}{2(n-1)v/r} \doteqdot (1+v^2)\phi(v/r).
\end{equation}
We claim that the solution is entire. Indeed,
define $v_+ \doteqdot \frac{r}{\sqrt{(n-1)(n-2)}}$. Then $v_+'=\frac{1}{\sqrt{(n-1)(n-2)}} >0=(1+v_+^2)\phi(v_+/r)$, so that $v_+$ is a supersolution. By a similar calculation, $v_-\doteqdot \frac{r}{\sqrt{n(n-1)}}$ is a subsolution. Both these are defined for all $r\geq0$. Therefore by Theorem \ref{extensibility}, $v$ is defined for $r \geq 0$. Recovering $u$ using (\ref{uformula}), we see that the bowl-type soliton is entire.

\subsubsection{Gauss curvature flows}
The examples we provide here have a speed function of the form $K^{\frac{\alpha}{n}}$, where $K$ is the Gauss curvature, i.e. $f(x_1,...,x_n)=(x_1 \cdot ... \cdot x_n)^{\frac{\alpha}{n}}$, which is an $\alpha$-homogeneous function. Then
\begin{align*}
    &f(x,y \mathbf{e})=z\\
    \iff &(xy^{n-1})^{\alpha/n}=z\\
    \iff &x=\frac{z^{\frac{n}{\alpha}}}{y^{n-1}}\,.
\end{align*}

Thus the resulting ODE is of the form
\begin{equation}
    v'=\frac{r^{n-1}(1+v^2)^{(n+2)/2-1/(2{\alpha/n})}}{v^{n-1}}
\end{equation}

This equation is separable, and by comparing with Lemma \ref{ODElemma}, one sees that $v$ blows up at $r=R$ for some finite $R$ precisely when ${\alpha/n} > 1/2$, and $v$ exists for all $r$ when ${\alpha/n} \leq 1/2$. Thus rotational translators of $K^{\alpha/n}$ flows exists over bounded domains and are asymptotic to a cylinder precisely when $\alpha >n/2$ and are entire when $\alpha \leq n/2$.

\section{Existence and uniqueness}

Our aim is to show that there exists a unique solution to the problem
\begin{equation} \label{IVP}
\begin{split}
v'={}&(1+v^2)^{1+\beta}g\left(\frac{v}{r(1+v^2)^\beta},1\right)\,,\\
v(0)={}&0
\end{split}
\end{equation}
on a maximal interval $[0,R)$, $0<R\le\infty$, which is of class $C^1([0,R))$. Note that this is non-trivial since the problem is singular at $r=0$.

We shall obtain a solution to \eqref{IVP} as the limit of a sequence of solutions $v_n$ to the approximating problems
\begin{equation*}
\begin{split}
v_n'={}&(1+v_n^2)^{1+\beta}g\left(\frac{v_n}{r(1+v_n^2)^\beta},1\right)\,,\\
v_n(r_n)={}&a_n
\end{split}
\end{equation*}
with initial values  $(r_n,a_n)\to (0,0)$.

Our approach is to solve the equation near the origin on some small interval $[0,\delta]$ where $\delta$ will be determined later. We first obtain a subsolution and some supersolutions to the ODE which will serve as uniform lower and upper barriers on $[0,\delta]$. Note that if $v$ satisfies the initial condition and admits a (one-sided) derivative at $r=0$, then its derivative must satisfy $v'(0)=1/f(1,...,1)^{1/\alpha}$. We verify  this by allowing $r \to 0$ in (\ref{IVP}) and observe the following:
\begin{align*}
    & v'(0)=g(v'(0),1) \\
\iff & f(v'(0),v'(0)\mathbf{e})=1 \\
\iff & v'(0)^\alpha f(1,...,1)=1 \\
\iff & v'(0)=1/f(1,...,1)^{1/\alpha}\,.
\end{align*}
So we define $\gamma \doteqdot 1/f(1,...,1)^{1/\alpha} $. Then $\gamma$ is the unique solution to $\gamma =g(\gamma,1)$ due to the above calculation.

Now define a function $w(r)$ implicitly by the relation
\begin{equation} \label{levelseteqn}
    \frac{w}{r(1+w^2)^\beta}=\gamma\,.
\end{equation}
Note that $w$ is well-defined by Lemma \ref{lem:level set invertibility}.

\begin{proposition} \label{subsol}
   The function $w$ as defined in \eqref{levelseteqn} is a subsolution to \eqref{IVP}.
\end{proposition}

\begin{proof}
    Due to (\ref{dr/dw}),
    \begin{equation}
    \begin{split}
        w' &= \gamma \frac{(1+w^2)^{1+\beta}}{1+(1-2\beta)w^2}\\
            &=g(\gamma,1)\frac{(1+w^2)^{1+\beta}}{1+(1-2\beta)w^2} \\
            &\leq (1+w^2)^{1+\beta}g(\gamma,1)\\
            &=(1+w^2)^{1+\beta}g\left(\frac{w}{r(1+w^2)^\beta},1\right)\,.\qedhere
    \end{split}
\end{equation}
\end{proof}

Given that the equation $f(x,\gamma\mathbf{e})=1$ has a solution $x=\gamma$, the implicit function theorem now guarantees that the equation $f(x,y\mathbf{e})=1$ can be solved for $x$ when $y \in [\gamma e^{-\theta}
,\gamma e^\theta]$, for some $\theta>0$. In other words, $ [\gamma e^{-\theta},
\gamma e^\theta]$ is in the domain of $g(\cdot,1)$. Given $\epsilon \in [0,\theta]$, let $\gamma_\epsilon= \gamma e^\epsilon$, and define $w_\epsilon(r)$ 
using  $\frac{w_\epsilon}{r(1+w_\epsilon^2)^\beta}=\gamma_\epsilon$.

\begin{proposition}
    For each $\epsilon>0$, there exists $r_\epsilon>0$ such that $w_\epsilon$ is a supersolution to (\ref{IVP}) on $(0,r_\epsilon)$.
\end{proposition}

\begin{proof}
Note that 
\[
w_\epsilon'= \gamma_\epsilon \frac{(1+w_\epsilon^2)^{1+\beta}}{1+(1-2\beta)w_\epsilon^2}
\]
and
\[
(1+w_\epsilon^2)^{1+\beta} g\left(\frac{w_\epsilon}{r(1+w_\epsilon^2)^\beta},1\right)= (1+w_\epsilon^2)^{1+\beta} g(\gamma_\epsilon,1)\,.
\]
Therefore,
\begin{align*}
    & w_\epsilon' > (1+w_\epsilon^2)^{1+\beta} g\left(\frac{w_\epsilon}{r(1+w_\epsilon^2)^\beta},1\right)\\
    \iff & \gamma_\epsilon \frac{(1+w_\epsilon^2)^{1+\beta}}{1+(1-2\beta)w^2} > (1+w_\epsilon^2)^{1+\beta} g(\gamma_\epsilon,1)\\
    \iff &\frac{\gamma e^{\epsilon}}{1+(1-2\beta)w_\epsilon^2} > g(\gamma e^{\epsilon},1)\\
    \iff &\gamma^\alpha e^{\alpha\epsilon}f\left(\frac{1}{1+(1-2\beta)w_\epsilon^2},\mathbf{e}\right) > 1\\
    \iff &f\left(\frac{1}{1+(1-2\beta)w_\epsilon^2},\mathbf{e}\right) > \frac{f(1,\mathbf{e})}{e^{\alpha\epsilon}}\,.
\end{align*}

The above inequality holds when $r=0$, hence it holds by continuity on some $[0,r_\epsilon]$, where $r_\epsilon$ might depend on $\epsilon$.
\end{proof}

It is important to note that, by Lemma \ref{lem:level set invertibility epsilon}, for any fixed $r$, $w_\epsilon(r)$ is monotone increasing with respect to $\epsilon$. In particular, $w_\theta \geq w_0$.

Now we construct a family of functions that converge to a solution of our ODE on $[0,r_\theta]$. For each $n>1/r_\theta$, consider the continuous function $v_n$ defined as follows.
\begin{itemize}
    \item $\frac{v_n}{r(1+v_n^2)^\beta}=\gamma$ on $[0,1/n]$
    \item $v_n$ obeys the ODE $v_n'=(1+v_n^2)^{1+\beta}g\left(\frac{v_n}{r(1+v_n^2)^\beta},1 \right)$ on $(1/n,r_\theta]$.
\end{itemize}

The initial data for the ODE is of course $v_n(1/n)$, which is implicitly defined by the first relation. Note that $v_n$ is well-defined on the interval $(1/n,r_\theta]$ due to Theorem \ref{extensibility}, where the compact set can be taken to be 
\[
K=\left\{(r,w) \in \mathbb{R}^2: \gamma \leq \frac{w}{r(1+w^2)^\beta} \leq \gamma_1, 0 < r \leq r_\theta\right\} \cup \{(0,0)\}\,.
\]

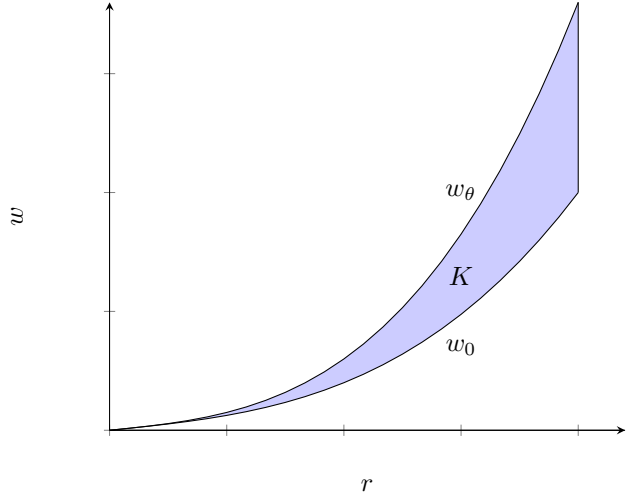
\begin{center}
    \begin{tikzpicture}
  \begin{axis}
    [
    axis lines = left,
    xlabel = \(r\),
    ylabel = {\(w\)},
    xticklabels=\empty,
    yticklabels=\empty
    ]
    \addplot[domain=0:2.2]{0};
    \addplot[name path=f, 
    domain=0:2
    ] {x+x^3}; 
    \addplot[name path=g,
    domain=0:2
    ] {x+2*x^3};
    
    \node at (axis cs:1.5,10) {$w_\theta$};
    \node at (axis cs:1.5,3.5) {$w_0$};
    \node at (axis cs:1.5,6.5) {$K$};
    \addplot [
        fill=blue, 
        fill opacity=0.2
    ]
    fill between[
        of=f and g
    ];
    \addplot [
    samples=100, 
    smooth,
    domain=0:2,
    black] coordinates {(2,10)(2,18)};
  \end{axis}
\end{tikzpicture}    
\end{center}

Now let us define a constant $C$ as follows. First, we define the following constants:
\begin{align*}
    C_1 &\doteqdot \sup \left\{(1+v^2)^{1+\beta}g\left(\frac{v}{r(1+v^2)^\beta},1\right):(r,v) \in K-\{(0,0)\}\right\}\\
    C_2 &\doteqdot \sup \left\{2\beta v(1+v^2)^\beta g\left(\frac{v}{r(1+v^2)^\beta},1\right): (r,v) \in K-\{(0,0)\} \right\}\\
    C_3 &\doteqdot \max \{C_1, C_2\}\\
    C_4 &\doteqdot e^{C_3 r_\theta}
\end{align*}
$C_1,C_2$ are finite due to the fact that $v,\frac{v}{r(1+v^2)^\beta}$ are bounded in $K-\{(0,0)\}$. Hence all of the constants above are finite.

Now define $C\doteqdot \max \{C_i: i=1,...,4\}$.
\begin{proposition} \label{unifconv}
    The sequence of functions $\{v_n\}$ converges uniformly on $(0,r_\theta]$.
\end{proposition}

\begin{proof}
    For $m>n$, we apply the mean value theorem to obtain the following estimate:
\begin{align*}
    v_m(1/n) &\leq v_m(1/m)+(1/n-1/m) \sup_{r \in [1/m,1/n]}v_m'(r)\\
    &\leq v_m(1/m)+(1/n)C_1\,.
\end{align*}
Thus,
\begin{equation} \label{eq:vmvnestimate}
   v_m(1/n)-v_n(1/n) \leq v_m(1/m)-v_n(1/n)+C_1/n\,.
\end{equation}

Now, from the definition of $v_n$, $v_n(1/n)=w(1/n)$ and $v_m(1/m)=w(1/m)$, where $w$ is given by $\frac{w}{r(1+w^2)^\beta}=\gamma$. Since $w$ is monotone increasing in $r$,
\begin{equation*}
    v_m(1/m)-v_n(1/n) \leq 0\,.
\end{equation*}
Thus (\ref{eq:vmvnestimate}) becomes
\begin{equation*}
    v_m(1/n)-v_n(1/n) \leq C_1/n\,.
\end{equation*}

Now we estimate $v_m(r)-v_n(r)$ for $r \in (1/n,r_\theta]$. Regarding the right hand side of (\ref{eqn:alphaode}) as a function of $v$ and applying the mean value theorem, we get
\begin{align*}
    v_m'(r)-v_n'(r) &= (1+v_m^2)^{1+\beta}g\left(\frac{v_m}{r(1+v_m^2)^\beta},1\right) - (1+v_n^2)^{1+\beta}g\left(\frac{v_n}{r(1+v_n^2)^\beta},1\right)\\
    &\leq  (v_m-v_n) \sup_{v\in [v_n(r),v_m(r)]} \left\{ 2\beta v(1+v^2)^\beta g\left(\frac{v}{r(1+v^2)^\beta},1\right)\right.\\&\left.+(1+v^2)g_y\left(\frac{v}{r(1+v^2)^\beta},1\right)\frac{(1+(1-2\beta)v^2)}{r(1+v^2)^{\beta+1}} \right\}\\
    &\leq  (v_m-v_n)\sup_{v\in [v_n(r),v_m(r)]} \left\{2\beta v(1+v^2)^\beta g\left(\frac{v}{r(1+v^2)^\beta},1\right)\right\}\\
    &\leq C_2(v_m-v_n)\,.
\end{align*}
We used the fact that $g$ is decreasing in the first slot in the penultimate step.

Dividing by the positive quantity $v_m-v_n$ and integrating on $[1/n,r]$,
\begin{align*}
    v_m(r)-v_n(r) &\leq (v_m(1/n)-v_n(1/n)) e^{C_2(r-1/n)}\\
    &\leq C^2/n\,.
\end{align*}
The claim follows by letting $n\to \infty$.
\end{proof}

Therefore ${v_n}$ converges uniformly to a continuous function on $(0,r_\theta)$. Since the sequence of functions $\{v_n\}$ is uniformly bounded, $v_n'$  has a uniformly continuous dependence on $v_n$, and hence $v_n'$ converges uniformly as well. Thus we have $C^1$-convergence of ${v_n}$ to some differentiable function $v$ on $(0,r_\theta)$. In particular, $v$ satisfies (\ref{eqn:alphaode}).

\begin{proposition}
    The limit function $v$ can be extended continuously to $r=0$. Moreover, the extended function is differentiable at $r=0$.
\end{proposition}
\begin{proof}
    Recall that since we have a subsolution $w$ on $(0,r_\theta)$ and supersolutions $w_\epsilon=$ on $(0,r_\epsilon)$,
\begin{equation} 
    \gamma r \leq \frac{v_n}{(1+v_n^2)^\beta} \leq \gamma e^\epsilon r \text{  on } (0,r_\epsilon).
\end{equation}
Passing to the limit, we have
\begin{equation} \label{eqn: subsuper}
    \gamma r \leq \frac{v}{(1+v^2)^\beta} \leq \gamma e^\epsilon r \text{  on  } (0,r_\epsilon).
\end{equation}
Letting $r \to 0$, we see that $v(0)=0$. Also, dividing the same equation by $r$, we can evaluate $v'(0)= \lim_{r \to 0} v/r$.
\begin{equation*}
    \gamma \leq \frac{v/r}{(1+v^2)^\beta} \leq \gamma e^\epsilon \text{  on  } (0,r_\epsilon).
\end{equation*}
Letting $\epsilon \to 0$, we get
\[ v'(0) = \gamma\,. \]

Now we show differentiability at $r=0$. Estimating equation (\ref{IVP}) from above and below using (\ref{eqn: subsuper}),
\begin{equation}
    (1+w^2)g(\gamma e^\epsilon,1) \leq v'(r) \leq (1+w_\epsilon^2)g(\gamma,1)\,.
\end{equation}
Letting $\epsilon \to 0$,
\begin{equation}
    \lim_{r \to 0} v'(r)=g(\gamma,1)=\gamma=v'(0)\,.
\end{equation}
This completes the proof.
\end{proof}


Regarding the question of uniqueness, we are only interested in solutions $v$ of class $C^1$ up to the origin, because these correspond to $C^2$-solutions $u$ of (\ref{TODE}).

\begin{proposition}
    The solution $v$ as obtained above is the unique $C^1$-solution to the initial value problem (\ref{IVP}).
\end{proposition}
\begin{proof}

    We first show uniqueness near the origin. The equation $f(x,y\mathbf{e})=1$ is solvable for $x$ when $y=\gamma$. Thus by the implicit function theorem, there exists $\epsilon>0$ such that this equation is solvable for $y \in (\gamma e^{-\epsilon}, \gamma e^\epsilon)$.
        Suppose that $v_1, v_2$ are two solutions, with initial conditions $v_1(0)=v_2(0)=0$. Now, since both solutions are $C^1$ up to the origin, $v_1'(0)=v_2'(0)=\gamma$, and hence there exists some $\delta>0$ such that the graphs of $v_1, v_2$ are contained in the compact set $K' \doteqdot \{(r,v): \gamma e^{-\epsilon}r \leq v \leq \gamma e^{\epsilon}r, \, 0 \leq r\leq \delta \}$ The slope field is nonsingular in $K'-\{(0,0)\}$, and hence integral curves do not intersect. Thus without loss of generality, we may assume $v_1>v_2$ in $K'-\{0\}$. Define $C'=\sup_{(r,v)\in K'-\{(0,0)\}}\left\{2\beta v(1+v^2)^\beta g \left(\frac{v}{r(1+v^2)^\beta},1\right)\right\}$.  Now let $\delta',r$ be real numbers such that $0<\delta'\leq r \leq \delta$. By an argument similar to  Proposition \ref{unifconv},
\begin{equation*}
    |v_2(r)-v_1(r)| \leq C'|v_2(\delta')-v_1(\delta')|
\end{equation*}
Since $v_1,v_2$ are continuous and agree at $r=0$, letting $\delta' \to 0$ shows that the solutions agree on $[0,\delta]$. This gives local uniqueness near the origin. Now by standard ODE theory, the solutions agree for as long as they are both defined.
\end{proof}

Due to Theorems \ref{existence} and \ref{extensibility}, we have proved the following.
\begin{lemma}
The initial value problem (\ref{IVP}) has a unique solution $v$ defined on some maximal interval $[0,R)$, where $0<R\leq\infty$. If $f$ is of class $C^k$, then $v$ is of class $C^{k+1}$ everywhere except possibly $r=0$. Moreover, $v$ and $v'$ are continuous up to $r=0$, with $v(0)=0$ and $v'(0)=\gamma$.
\end{lemma}

From this, we recover $u$ using the formula $u(r)=\int_0^r v(\rho) d\rho$. Note that $u$ is $C^2$ at $r=0$. Thus we have the following existence result for bowl-type solitons.

\begin{theorem}\label{thm:classical solution}
Let $f$ be an admissible speed. There exists a unique bowl-type soliton with velocity $e_{n+1}$ correspinding to it. The bowl-type soliton is the graph of a function $u:B_R\to\R$, where $0<R\leq\infty$.  If $f$ is of class $C^k$, then $u$ is of class $C^{k+2}$ everywhere except possibly the origin, and at least $C^2$ at the origin.
\end{theorem}

This proves, in particular, the first part of Theorem \ref{thm:entire}.

\section{Smoothness at the origin}

For higher regularity at the origin, we apply the PDE lemma (Proposition \ref{smoothnessthm}). So let us cast our bowl-type soliton (near $0$) as the solution to a (fully nonlinear) elliptic PDE.

Recall that, for a graph $M=\graph u$, the component matrix of the Weingarten tensor is given by
\[
W=g^{-1}\cdot\sff\,,
\]
where
\[
\sff=\frac{D^2u}{\sqrt{1+\vert Du\vert^2}}
\]
is the component matrix of the second fundamental form and 
\[
g^{-1}=\fff-\frac{Du\otimes Du}{1+\vert Du\vert^2}
\]
is the component matrix of the cometric. Since $W$ is not in general a symmetric matrix, we consider instead the matrix \cite{Ur91}
\[
\tilde W\doteqdot P\cdot \sff\cdot P\,,
\]
where $P$, a square root of $g^{-1}$, is given by
\[
P=\fff-\frac{Du\otimes Du}{\sqrt{1+\vert Du\vert^2}\big(1+\sqrt{1+\vert Du\vert^2}\big)}.
\]
Note that $\tilde W$ is symmetric and has the same eigenvalues as $W$. Thus, if $M$ is a translating solution to the flow by speed $f$, then $u$ is a solution to the equation
\begin{equation}\label{graphical translator equation}
-\hat f\left(Du,D^2u\right)=-\frac{1}{\sqrt{1+\vert Du\vert^2}}\,,
\end{equation}
where $\hat f:\R^n\times S_+^{n\times n}\to\mathbb R$ ($S_+^{n\times n}$ are the positive definite symmetric $n\times n$ matrices) is defined by
\[
\begin{split}
{}&\hat f(p,r)\doteqdot\\
{}&f\left(\!\left(\fff-\frac{p\otimes p}{\sqrt{1+\vert p\vert^2}\big(1+\sqrt{1+\vert p\vert^2}\big)}\right)\!\cdot\! \frac{r}{\sqrt{1+|p|^2}}\!\cdot\!\left(\fff-\frac{p\otimes p}{\sqrt{1+\vert p\vert^2}\big(1+\sqrt{1+\vert p\vert^2}\big)}\right)\!\right),
\end{split}
\]
where we treat $f$ as a function of a symmetric matrix $Z$ by evaluating it on the eigenvalues $z_1,\dots,z_n$ of $Z$.

Observe that $\hat{f}$ is of the same smoothness class as $f$ and
\begin{align*}
    \left(\frac{\partial \hat{f}}{\partial r_{ij}}\right)_{(Du,D^2u)}={}&\left(\frac{\partial f}{\partial r_{pq}}\right)_{(P \cdot \sff \cdot P)}P_{pi}P_{qj} \\
    \implies\;\; \frac{\partial \hat{f}}{\partial r_{ij}} \xi_i \xi_j={}&\frac{\partial f}{\partial r_{pq}}P_{pi}P_{qj}\xi_i \xi_j >0
\end{align*}
for all $(\xi_1,...,\xi_n) \in \mathbb{R}^n - \{0\}$, because $P$ is non-degenerate and the eigenvalues of the matrix $\left(\frac{\partial f}{\partial r_{pq}}\right)$, which are equal to $f_{z_i}$, $i=1,...,n$, are positive by hypothesis. This implies that \eqref{graphical translator equation} is elliptic

Finally, since we have proved that the solution $u$ corresponding to our bowl-type soliton is of class $C^2$, Proposition \ref{smoothnessthm} yields the following improvement of Theorem \ref{thm:classical solution}.

\begin{theorem}
There exists a unique bowl-type soliton for every admissible speed $f$. The bowl-type soliton is the graph of a function $u:B_R\to\R$, where $0<R\leq\infty$. If $f$ is of class $C^{k,\alpha}$, then $u$ is of class $C^{k+2,\alpha}$. In particular, if $f$ is smooth, then so is $u$.
\end{theorem}

This proves the second part of Theorem \ref{thm:entire}.

\section{Low homogeneities}
Recall that regardless of $\alpha$, (\ref{subsol}) shows that the function $v_-$ implicitly defined by $\frac{v_-}{r(1+v_-^2)^\beta}=\gamma$ is a subsolution to the translator ODE. From \S \ref{levset} we infer that the solution $v$ satisfies $\frac{v}{r(1+v^2)^\beta} \geq \gamma >0$. In addition, $v\geq 1$ for sufficiently large $r$.

Now if $\alpha \leq 1/2$, we have that for sufficiently large $r$ (such that $v \geq 1$),
\begin{align*}
    v' &\leq \gamma (1+v^2)^{1+\beta}\\
    & \leq \gamma \sqrt{1+v^2}\\
    & \leq \sqrt{2} \gamma v \,.
\end{align*}
By comparing with Lemma \ref{ODElemma}, we see that the corresponding bowl-type soliton is entire.

This proves Theorem \ref{thm:low homogeneity}.

\section{Degenerate speeds}
We only need to discuss what happens if $\alpha >1/2$. We formulate the degeneracy condition $f(0,\mathbf{e})=0$ as  $\lim_{s \to 0}$ $f(s,\mathbf{e})=0$, as $f$ may be undefined when one of its inputs is zero. Note that in the equation $f(x,y\mathbf{e})=1$, $x$ is decreasing with respect to $y$. If $x \to 0$ as $y\to y_0 < \infty$, then $\lim_{y \to y_0}f(0,y\mathbf{e})=1$ which violates our degeneracy hypothesis that $f(0,\mathbf{e})=0$. Thus $L \doteqdot \lim_{y \to \infty}x=\lim_{y \to \infty}g(y,1) \geq 0$

\subsection{The $L>0$ case}
If $L>0$, we have the inequality
\begin{align*}
    v'&>(1+v^2)^{1+\beta}L\\
    & > Lv^{2+2\beta}\,.
\end{align*}
Comparing with Lemma \ref{ODElemma} shows that solutions are defined on bounded domains provided $\beta>-1/2$, i.e. $\alpha>1/2$. This explains what we observed in the harmonic mean curvature case, and proves Theorem \ref{L>0}.

\subsection{The $L=0$ case}
Again we analyse the case of $\alpha >1/2$. Here it turns out that whether the bowl-type soliton is entire or nonentire depends on the the asymptotics of $g(y,1)$ as $y \to \infty$.

First, we note that since the function $v_-(r)$ implicitly defined by $\frac{v_-}{r(1+v_-^2)^\beta}=\gamma$ is a subsolution, the solution $v$ satisfies $v \geq v_-$, and hence $\frac{v}{r(1+v^2)^\beta} \geq \gamma >0$. We claim that $\frac{v}{r(1+v^2)^\beta}$ is in fact unbounded. We prove this by contradiction. Suppose this is not the case. Then there exists $M>0$ such that $\frac{v}{r(1+v^2)^\beta} \leq M$. This means $v(r)$ exists for all $r$. Define $g(M,1) \doteqdot \epsilon$. Then, since $L=0$ and $g(\cdot,1)$ is strictly decreasing, $\epsilon>0$. Since $g$ is monotone decreasing in the first argument, $g\left(\frac{v}{r(1+v^2)^\beta},1\right)\geq \epsilon$. This means
\begin{equation*}
    v' \geq (1+v^2)^{1+\beta} \epsilon
\end{equation*}
but since $\alpha>1/2$ (i.e. $\beta>-1/2$), this equation blows up at some finite $R$, leading to a contradiction. This proves our claim.
Indeed more is true: given any $N>0$, there exists $r_1>0$ such that $\frac{v}{r(1+v^2)^\beta}\geq N$ for $r \geq r_1$. This is because of the following lemma which holds for both degenerate and nondegenerate speeds $f$. The ``$\asymp$" symbol used below,, which is an equivalence relation on the asymptotics of two functions, is defined in \S \ref{levset}.

\begin{lemma}\label{eventualsub}
Suppose $N>0$ and that $g(N,1)>0$. Then the function $w_N$ defined implicitly by $\frac{w_N}{r(1+w_N^2)^\beta}=N$ is a subsolution to the translator ODE (\ref{eqn:alphaode}) for sufficiently large $r$.
\end{lemma}
\begin{proof}
For large $r$, by (\ref{dw/dr}) we have $w_N' \asymp r^{\alpha-1}$. On the other hand, the asymptotics of the right hand side of the ODE is
\begin{align*}
    (1+w_N^2)^{1+\beta} g\left(\frac{w_N}{r(1+w_N^2)^\beta},1\right) &= \epsilon_N (1+w_N^2)^{1+\beta}\\
    & \asymp (r^{2\alpha})^{1+\beta}\\
    &= (r^{2\alpha})^{3/2-1/2\alpha}\\
    &= r^{3\alpha-1}
\end{align*}
where $\epsilon_N \doteqdot g(N,1) >0$. Thus for sufficiently large $r$, $w_N$ is a subsolution to (\ref{eqn:alphaode}).
\end{proof}

Now we can consider the asymptotics of $g(y,1)$ as $y \to \infty$.

Suppose $g(y,1)=O(y^{1-2\alpha})$, i.e. there exists a constant $C>0$ such that $g(y,1) \leq C y^{1-2\alpha} $ for sufficiently large $y$. Then for sufficiently large values of $r$,

\begin{align*}
    v' &= (1+v^2)^{1+\beta} g\left(\frac{v}{r(1+v^2)^\beta},1\right)\\
    & \leq C(1+v^2)^{1+\beta} \left(\frac{r(1+v^2)^\beta}{v}\right)^{2\alpha-1}\\
    &\leq  C' v r^{2\alpha-1}\,.
\end{align*}
Now, by comparing this to Lemma \ref{ODElemma} we see that $v$ exists for all $r>0$. In this case, the bowl-type soliton is entire.

In contrast with the previous case, now suppose there exists $C>0$ and some $k<2\alpha-1$ such that $g(y,1)>Cy^{-k}$. Define the positive number $\epsilon$ by the relation $\frac{\epsilon}{1-2\beta}=2\alpha-1-k$. Then,
\begin{align*}
    v' &= (1+v^2)^{1+\beta} g\left(\frac{v}{r(1+v^2)^\beta},1\right)\\
    & \geq C(1+v^2)^{1+\beta} \left(\frac{r(1+v^2)^\beta}{v}\right)^{2\alpha-1-\epsilon/(1-2\beta)}\\
    & \geq C' v^{1+\epsilon}\,.
\end{align*}
Comparing this to Lemma \ref{ODElemma}, one sees that the solution $v$ blows up at some finite $r=R$. In this case, the bowl-type soliton exists over the ball $B_R$.

This proves Theorem \ref{L>0 refined}.

\begin{remark}
    As mentioned in the introduction, this analysis leaves out the case when $x = O(y^{-k})$ for $k<2\alpha-1$ but not for $k=2\alpha-1$. This is because it is not possible to decide just using this condition in this boundary case whether the solution is entire or over a bounded domain. Consider for $p>0$ the differential equation $v'=v (\log v)^p$ with initial condition $v(r_0)=v_0>1$. The right hand side is $O(v^\theta)$ for $\theta >1$ but not $\theta=1$. This equation blows up at some finite $r$ if and only if $p>1$. The author suspects that a more general integrability condition on $g(\cdot,1)$ could be used to provide a complete classification, but since typical applications involve algebraic functions of the principal curvatures and transcendental functions are extremely rare, the criteria provided here are more readily applicable.
\end{remark}

\section{Nondegenerate speeds}

Finally, we prove the asymptotic expansion for bowl-type solitons of flows by nondegenerate speeds.
\subsection{Entireness}
As done in previous section, we can extend $f$ to the boundary $\partial \Gamma^n_+$ of the positive cone by taking limits. Suppose $f(0,\mathbf{e})>0$ as in Theorem \ref{thm:entire}. Let $\gamma'\doteqdot f(0,\mathbf{e})$. Then $g_1(0,1)=1/(\gamma')^{1/\alpha}$ is finite as well (Refer to \S \ref{sec:TODE} for the definition of $g_1$.)  Let $\gamma_+\doteqdot g_1(0,1)$. Then we have that $g_1(0,1)=\gamma_+ \iff g(\gamma_+,1)=0$. Now, the function $v_+(r)$ defined by $\frac{v_+}{r(1+v_+^2)^\beta}=\gamma_+$ is  precisely where the slope field vanishes. Also we see that $v_+$ is a supersolution to (\ref{eqn:alphaode}) because
\begin{align*}
    v_+' &> 0\\
         &=(1+v_+^2)^{1+\beta}g\left(\frac{v_+}{r(1+v_+^2)^\beta},1\right)\,.
\end{align*}

We showed in Proposition \ref{subsol} that the function $v_-(r)$ defined by $\frac{v_-}{r(1+v_-^2)^\beta}=\gamma$, where $\gamma \doteqdot 1/f(1,...,1)^{1/\alpha}$, is a subsolution to (\ref{eqn:alphaode}). Since we have subsolution and a supersolution, both defined on $[0,\infty)$, it follows from Theorem \ref{extensibility} that the solution $v$ is also defined on $[0,\infty)$. Using the formula $u(r)=\int_0^r v(\rho) d\rho$, we see that $u$ is defined on $[0,\infty)$. Therefore we have the following theorem.
\begin{theorem}
If $f(0,\mathbf{e})>0$, then the corresponding bowl-type soliton is entire.
\end{theorem}

\subsection{Asymptotics}
Let $v$ be the solution to (\ref{IVP}).  We prove the following 
\begin{proposition}
Suppose $f$ is a nondegenerate speed. Let $v$ be the solution to (\ref{IVP}). $\frac{v}{r(1+v^2)^\beta} \to \gamma_+$ as $r \to \infty$
\end{proposition}
\begin{proof}
Let $\epsilon>0, r_0>0$. We claim that there exists $ r_1 >r_0$ such that $ \frac{v(r_1)}{r_1(1+v(r_1)^2)^\beta} \geq (1-\epsilon)\gamma_+$.

Suppose this was not the case, i.e. for some $\epsilon>0$ we have $ \frac{v}{r(1+v^2)^\beta} <(1-\epsilon)\gamma_+$ for all $ r>r_0$ . Note that since $g$ is decreasing in the first slot, $g\left(\frac{v}{r(1+v^2)^\beta},1\right)>g((1-\epsilon)\gamma_+,1) \doteqdot  \epsilon'$. Then we have
\begin{equation*}
v'>(1+v^2)^{1+\beta}\epsilon'>(v^2)^{1+\beta}\epsilon'=v^{3-1/\alpha}\epsilon'    
\end{equation*}
which implies that
\begin{equation} \label{diffineq}
    v^{1/\alpha-3}v'>\epsilon'\,.
\end{equation}
If $\alpha >1/2$, then $3-1/\alpha>1$ so that $v$ blows up at some finite $R$. So we focus on the case that $\alpha \in (0,1/2]$. Let $r_0$ be any positive number, and define $v_0 \doteqdot v(r_0)$.

In case $\alpha=1/2$, the differential equality (\ref{diffineq}) becomes 
\[
v'/v>\epsilon'
\;\;\implies\;\; v>v_0 e^{r-r_0}\,.
\]
If $\alpha \in (0,1/2)$ then the differential inequality becomes

\begin{align*}
&\frac{(v^{1/\alpha-2})'}{1/\alpha-2} > \epsilon'\\
\implies\;\; &v(r)^{1/\alpha-2}-v_0^{1/\alpha-2}>(1/\alpha-2)\epsilon'(r-r_0)\\
\implies\;\; &v>[C+C'(r-r_0)]^\frac{1}{1/\alpha-2}\,,
\end{align*}
where $C=v_0^{1/\alpha-2}, C'=(1/\alpha-2)\epsilon'$.

Note that
\[
\frac{1}{1/\alpha-2}>\alpha \iff 1>\alpha(1/\alpha-2)=1-2\alpha \iff \alpha >0
\]
provided $\alpha <1/2$.

We have already shown that $\frac{v}{r(1+v^2)^\beta}=(1-\epsilon)\gamma_+$ implies $v$ grows like $r^\alpha$. This is in contradiction to what we have just shown, which is that $v$ grows like $r^{1/ \alpha-2}$ when $\alpha \in (0,1/2)$ and like $e^r$ when $\alpha=1/2$, both of which are strictly faster than $r^\alpha$.

Having proved this, now we conclude using Lemma \ref{eventualsub} that  for sufficiently large $r$, $\frac{v}{r(1+v^2)^\beta}$ stays above $(1-\epsilon)\gamma_+$.

This means that our solution $v$ exceeds $w$ as defined here, and due to upward monotonicity of the expression $\frac{w}{r(1+w^2)^\beta}$ with respect to $w$, we infer that 
\[
\frac{v}{r(1+v^2)^\beta}>\frac{w}{r(1+w^2)^\beta}=(1-\epsilon)\gamma_+
\]
The claim follows since $\epsilon$ is arbitrary.
\end{proof}
Therefore, we have that $v=\frac{r^\alpha}{f(0,\mathbf{e})}+ o(r^\alpha)$. Integrating, we see that the bowl-type soliton has the following asymptotics as $|x|\to \infty$:
\begin{equation*}
    u(|x|)=\frac{|x|^{\alpha+1}}{(\alpha+1)f(0,\mathbf{e})}+ o(|x|^{\alpha+1})\,.
\end{equation*}
This proves the asymptotic expansion that was asserted in Theorem \ref{thm:entire}.

\section{Convexity of solutions}
We show here that the bowl-type solitons that we have constructed are convex. Observe that for any admissible speed function, the function $v_-(r)$ defined by $\frac{v_-}{r(1+v_-^2)^\beta}=\gamma$, where $\gamma \doteqdot 1/f(1,...,1)^{1/\alpha}$ is a subsolution. For degenerate speeds, $\lim_{y \to \infty}g(y,1) \geq 0$ and for nondegenerate speeds the function $v_+(r)$ defined by $\frac{v_+}{r(1+v_+^2)^\beta}=\gamma_+$, where $\gamma_+\doteqdot g_1(0,1)$, is a supersolution. Therefore, for our solution $v$, the expression  $g\left(\frac{v}{r(1+v^2)^\beta},1\right)$
is positive in both the degenerate and nondegenerate cases for the following reason: $v$ is below this supersolution in the nondegenerate case, and in the degenerate case, $g(\cdot,1)$ is positive for all positive inputs. Therefore in either case, the solution satisfies $v'=(1+v^2)^{1+\beta} g\left(\frac{v}{r(1+v^2)^\beta},1\right)>0$, which implies that $u''>0$ for the profile curve $u$. Thus, $\kappa_1=\frac{u''}{(1+u'^2)^{3/2}}$. The remaining curvatures $\kappa_i=\frac{u'}{r\sqrt{1+u'^2}}$ are also positive because our subsolution guarantees that $u'=v\geq v_>0$. We conclude that the bowl-type solitons are convex.

\bibliographystyle{plain}
\bibliography{references}

\end{document}